\documentclass[12pt]{amsart}
\usepackage{amssymb, verbatim,url,lineno}
\usepackage{enumerate,mathtools}
\usepackage{color}
\usepackage{leftidx}

\newcommand{\fiab}[2]{\prescript{\text{AB}}{#1} I_{t}^{#2}}

\newcommand{\fdabc}[2]{\prescript{\text{ABC}}{#1} D_{t}^{#2}}
\newcommand{\fdabr}[2]{\prescript{\text{ABR}}{#1} D_{t}^{#2}}
\newcommand{\fdabry}[2]{\prescript{\text{ABR}}{#1} D_{y}^{#2}}
\newcommand{\fdabcy}[2]{\prescript{\text{ABC}}{#1} D_{y}^{#2}}

\date{\today}

\newtheorem{theorem}{Theorem}[section]

\theoremstyle{definition}
\newtheorem{definition}[theorem]{Definition}

\theoremstyle{remark}
\newtheorem{remark}{Remark}

\numberwithin{equation}{section}

\usepackage[T1]{fontenc}

\begin{document}

\title[Numerical computation of a fractional derivative]
{Numerical computation of a fractional derivative with non-local and non-singular kernel}

\author[Djida]{J.D. Djida}
\address[Djida]{African Institute for Mathematical Sciences (AIMS), P.O. Box 608, Limbe Crystal Gardens, South West Region, Cameroon. }
\email[Djida]{jeandaniel.djida@aims-cameroon.org}
\author[Area]{I. Area}
\address[Area]{Departamento de Matem\'atica Aplicada II,
              E.E. Telecomunicaci\'on,
              Universidade de Vigo,
              Campus Lagoas-Marcosende,
              36310 Vigo, Spain.}
\email[Area]{area@uvigo.es}
\author[Atangana]{A. Atangana}
\address[Atangana]{Institute for Groundwater Studies, Faculty of Natural and Agricultural Sciences, University of the Free State, 9301, Bloemfontein, South Africa.}
\email[Atangana]{abdonatangana@yahoo.fr}

\thanks{The first author is indebted to the AIMS-Cameroon 2015--2016  tutor fellowship.}


\begin{abstract}
A numerical scheme for solving fractional initial value problems involving the Atangana-Baleanu fractional derivative is presented.  
Some examples for the proposed method are included, both for equations and systems of fractional initial value problems.
\end{abstract}

\maketitle

\section{Introduction}
For a long period the fractional calculus, that is derivatives and integrals of nonnatural order, was developed as a purely theoretic field \cite{MR0361633,samko1993fractional}. Very recently, it has been shown that it can be used to explain certain physical problems, and also for processes where memory effects are important \cite{memory}. In this direction, while the classical derivative gives us the instantaneous rate of change of a function, the parameter of the fractional derivative can be understood as a memory index of the variation of the function, taking into account the previous instants. For this reason, in last years, fractional calculus has been fruitfully applied to different fields \cite{MR3323914,machado2010}, and in particular to epidemiological models \cite{sulami,nosebola,nosebola2}. It is also interesting to mention some recent works in order to find appropriate fractional analogues of the so-called special functions \cite{djida,ariane,Klimek2014402}.

\medskip
There exist various definitions ---Riemann, Liouville, Caputo, Grunwald-Letnikov, Marchaud, Weyl, Riesz, Feller, and others--- for fractional derivatives and integrals, (see e.g. \cite{Richard,Hilfer,MR0361633,samko1993fractional} and references therein). This diversity of definitions is due to the fact that fractional operators take different kernel representations in different function spaces.
In a recent work Caputo and Fabrizio \cite{capfab} introduced a new fractional derivative, analyzed e.g. in \cite{losnieto}. Moreover, in \cite{atanbaleanu} another fractional derivative with non-local and non-singular kernel was proposed. 

\medskip The main aim of this article is to present a numerical scheme for solving fractional initial value problems involving the fractional derivative introduced by Atangana and Baleanu \cite{atanbaleanu}.

\medskip The paper is structured as follows. In Section \ref{section:2} the basic definitions and notations are introduced, including the Atangana-Baleanu fractional derivatives and integral. In Section \ref{section:3} the relationship between Atangana-Baleanu derivative in Riemann-Liouville sense and Atangana-Baleanu fractional integral is derived. Finally, in Section \ref{section:4} a numerical scheme for solving fractional initial value problems (and fractional systems) involving the Atangana-Baleanu fractional derivatives is proposed, and some numerical examples are also included.

\section{Basic definitions and notations}\label{section:2}

The exponential function, $\exp(t)$, plays a fundamental role in mathematics and it is really useful in the theory of integer order differential equations. In the case of fractional order, the Mittag-Leffler function appears in a natural way.
\begin{definition}
The function $E_\alpha(z)$ is defined as
\begin{equation}\label{eq:mlf}
E_\alpha(z)=\sum_{k=0}^\infty\dfrac{z^k}{\Gamma(\alpha k+1)},\qquad\alpha>0.
\end{equation}
\end{definition}
This function provides a simple generalization of the exponential function because of the replacement of $k!=\Gamma(k+1)$ by $(\alpha k)!=\Gamma(\alpha k+1)$ in the denominator of the terms of the exponential series. Due to this, such function can be considered the simplest nontrivial generalization of exponential function.

Let $(a,b) \subset {\mathbb{R}}$ and let $v$ be a function of the Hilbert space $L^{2}(a,b)$. It can be identified to a distribution on $(a,b)$ as a function of $L^{1}_{\text{loc}}(a,b)$, also denoted as $v$ and we can define its derivative $v'$ as distribution on $(a,b)$. In general $v'$ is not an element of $L^{2}(a,b)$.
\begin{definition}
The Sobolev space of order $1$ on $(a,b)$ is defined as
\[
H^{1}(a,b)=\{ v \in L^{2}(a,b) \,\vert\, \,\, v' \in L^{2}(a,b) \}.
\]
\end{definition}

\begin{definition}
Let $f \in H^{1}(a,b)$, $b>a$, $\alpha \in [0,1]$. The Atangana-Baleanu fractional derivative of $f$ of order $\alpha$ in Caputo sense  with base point $a$ is defined at a point $t \in (a,b)$
\begin{equation}\label{eq:abfdc}
\fdabc{a}{\alpha} f(t)= \frac{B(\alpha)}{1-\alpha} \int_{a}^{t} f'(x) E_{\alpha} \left[ -\alpha \frac{(t-x)^{\alpha}}{1-\alpha} \right] dx.
\end{equation}
\end{definition}

\begin{definition}
Let $f \in H^{1}(a,b)$, $b>a$, $\alpha \in [0,1]$. The Atangana-Baleanu fractional derivative of $f$ of order $\alpha$ of $f$ in Riemann-Liouville sense with base point $a$ is defined at a point $t \in (a,b)$ as
\begin{equation}\label{eq:fdab}
\fdabr{a}{\alpha} f(t)=\frac{B(\alpha)}{1-\alpha} \frac{d}{dt} \int_{a}^{t} f(x) E_{\alpha} \left[ -\alpha \frac{(t-x)^{\alpha}}{1-\alpha} \right] dx.
\end{equation}
\end{definition}

\begin{definition}\label{def:abfi}
The Atangana-Baleanu fractional integral of order $\alpha$ with base point $a$ is defined as
\begin{equation}\label{eq:fi}
\fiab{a}{\alpha} f(t) = \frac{1-\alpha}{B(\alpha)}f(t) + \frac{\alpha}{B(\alpha) \Gamma(\alpha)} \int_{a}^{t} f(y) (t-y)^{\alpha-1} dy.
\end{equation}
\end{definition}
\begin{remark}
Notice that if $\alpha=0$ in \eqref{eq:fi} we recover the initial function, and if $\alpha=1$ in \eqref{eq:fi} we obtain the ordinary integral.
\end{remark}

\begin{remark}
If we impose that the average in \eqref{eq:fi} is equal to one,
\[
\frac{1-\alpha}{B(\alpha)}+ \frac{\alpha}{B(\alpha)\Gamma(\alpha)}=1,
\]
we obtain
\begin{equation}
B(\alpha)=1-\alpha +\frac{\alpha }{\Gamma (\alpha )}.
\end{equation}
\end{remark}

Let us recall the following result \cite[Theorem 3]{atbalchaos} concerning the Laplace transform of both fractional derivatives.
\begin{theorem}\label{theorem:lt}
The following Laplace transforms hold
\begin{align}
{\mathcal{L}} \left[\fdabc{0}{\alpha} f(t)  \right] &= \frac{B(\alpha)}{1-\alpha} \frac{p^{\alpha} {\mathcal{L}} \{ f(t)\}(p) - p^{\alpha-1} f(0)}{p^{\alpha}+\alpha/(1-\alpha)}, \label{laplace:caputo} \\
{\mathcal{L}} \left[ \fdabr{0}{\alpha} f(t) \right] &=\frac{B(\alpha)}{1-\alpha} \frac{p^{\alpha} {\mathcal{L}} \{f(t)\}(p)}{p^{\alpha}+\alpha/(1-\alpha)} .
\end{align}
\end{theorem}

\section{Further properties of Atangana-Baleanu fractional derivative}\label{section:3}

In this section we establish the relationship between Atangana-Baleanu derivative in Riemann-Liouville sense and Atangana-Baleanu fractional integral.

\begin{theorem}
Let $f \in H^{1}(a,b)$, $b>a$, such that the Atangana-Baleanu fractional derivative exists. Then, the following relations hold true
\begin{align}
\fiab{0}{\alpha} \left \{ \fdabr{0}{\alpha} f(t) \right \} = f(t), \label{eq:new1} \\
\fiab{0}{\alpha} \left \{ \fdabc{0}{\alpha} f(t) \right \} = f(t)-f(0). \label{eq:new2}
\end{align}
\end{theorem}
\begin{proof}
We establish the first relation \eqref{eq:new1} by using the Laplace transform. From Definition \ref{def:abfi} we have
\begin{equation}
\fiab{0}{\alpha} \left \{ \fdabr{0}{\alpha} f(t) \right \} = \frac{1-\alpha}{B(\alpha)} \fdabr{0}{\alpha} f(t) + \frac{\alpha}{B(\alpha)\Gamma(\alpha)} \int_{0}^{t} (t-y)^{\alpha-1} \fdabry{0}{\alpha} f(y)dy.
\end{equation}
By applying on both sides of the latter equation the Laplace transform, by using Theorem \ref{theorem:lt} we obtain
\begin{multline}
{\mathcal{L}} \left[ \fiab{0}{\alpha} \left \{ \fdabr{0}{\alpha} f(t) \right \} \right] 
\\= \frac{1-\alpha}{B(\alpha)} {\mathcal{L}} \left[ \fdabr{0}{\alpha} f(t) \right] 
+\frac{\alpha}{B(\alpha)\Gamma(\alpha)}   {\mathcal{L}} \left[ \int_{0}^{t} (t-y)^{\alpha-1} \fdabry{0}{\alpha} f(y)dy \right] \\
=\frac{1-\alpha}{B(\alpha)} \frac{B(\alpha)}{1-\alpha} \frac{s^{\alpha} F(s)}{s^{\alpha}+\frac{\alpha}{1-\alpha}}
+ \frac{\alpha}{B(\alpha)}\frac{B(\alpha)}{1-\alpha} s^{-\alpha} \frac{s^{\alpha} F(s)}{s^{\alpha}+\frac{\alpha}{1-\alpha}}
\\ =\frac{s^{\alpha} F(s)}{s^{\alpha}+\frac{\alpha}{1-\alpha}} + \frac{\alpha}{1-\alpha} \frac{F(s)}{s^{\alpha}+\frac{\alpha}{1-\alpha}} = F(s).
\end{multline}
We shall now obtain \eqref{eq:new2} by using again the Laplace transform. From Definition \ref{def:abfi} we have
\begin{equation}\label{eq:aux2}
\fiab{0}{\alpha} \left \{ \fdabc{0}{\alpha} f(t) \right \} = \frac{1-\alpha}{B(\alpha)} \fdabr{0}{\alpha} f(t) + \frac{\alpha}{B(\alpha)\Gamma(\alpha)} \int_{0}^{t} (t-y)^{\alpha-1} \fdabcy{0}{\alpha} f(y)dy.
\end{equation}
If we apply the Laplace transform to both sides of  \eqref{eq:aux2}, by using Theorem \ref{theorem:lt} it yields
\begin{multline}
{\mathcal{L}} \left[ \fiab{0}{\alpha} \left \{ \fdabc{0}{\alpha} f(t) \right \} \right] 
\\=\frac{1-\alpha}{B(\alpha)} {\mathcal{L}} \left[\fdabc{0}{\alpha} f(t) \right] + \frac{\alpha}{B(\alpha)\Gamma(\alpha)} {\mathcal{L}}  \left[ \int_{0}^{t} (t-y)^{\alpha-1} \fdabcy{0}{\alpha} f(y)dy \right] \\
= \frac{1-\alpha}{B(\alpha)} \frac{B(\alpha)}{1-\alpha} \frac{p^{\alpha} {\mathcal{L}} \{ f(t)\}(p) - p^{\alpha-1} f(0)}{p^{\alpha}+\alpha/(1-\alpha)} + \frac{\alpha}{B(\alpha)\Gamma(\alpha)} {\mathcal{L}}  \left[ \int_{0}^{t} (t-y)^{\alpha-1} \fdabcy{0}{\alpha} f(y)dy \right] \\
=\frac{p^{\alpha} F(p) - p^{\alpha-1} f(0)}{p^{\alpha}+\alpha/(1-\alpha)} + \frac{\alpha}{B(\alpha)} p^{-\alpha} \frac{B(\alpha)}{1-\alpha}  \frac{p^{\alpha} F(p) - p^{\alpha-1} f(0)}{p^{\alpha}+\alpha/(1-\alpha)}=F(p)-\frac{f(0)}{p}.
\end{multline}
\end{proof}

\section{Numerical solution of fractional initial value problems involving the Atangana-Baleanu fractional derivative in Caputo sense}\label{section:4}

Let us now consider the fractional initial value problem
\begin{equation}\label{eq:6abc}
\begin{cases}
\fdabc{0}{\alpha} y(t)=g(t,y(t)), \quad t \in [0,T], \\
y(0)=y_{0},
\end{cases}
\end{equation}
where $\fdabc{0}{\alpha} y(t)$ is defined in \eqref{eq:abfdc}. If we apply the fractional integral defined in \eqref{eq:fi} to \eqref{eq:6abc}  it yields
\begin{equation}
y(t)-y(0) = \frac{1-\alpha}{B(\alpha)}g(t,y(t)) + \frac{\alpha}{B(\alpha) \Gamma(\alpha)} \int_{0}^{t} g(s,y(s)) (t-s)^{\alpha-1} ds,
\end{equation}
by using \eqref{eq:new2}. Let us assume that \eqref{eq:6abc} has a unique solution on $[0,T]$, and consider the lattice $\{t_{n}=nh\}_{n=0}^{N}$ where $N$ is a nonnegative integer, with $h=T/N$. Then,
\begin{equation}
y(t_{n+1}) =y(0)+ \frac{1-\alpha}{B(\alpha)}g(t_{n+1},y(t_{n+1})) + \frac{\alpha}{B(\alpha) \Gamma(\alpha)} \int_{0}^{t_{n+1}} g(s,y(s)) (t_{n+1}-s)^{\alpha-1} ds.
\end{equation}
In a similar way as in \cite{Diethelm} for the Caputo fractional derivative, for the integral part, we use the product trapezoidal quadrature formula to replace the integral where the nodes $t_{j}$ are computed with respect to the weight function $(t_{n+1}-\cdot)^{\alpha-1}$. That is, 
\begin{equation}
\int_{0}^{t_{n+1}} (t_{n+1}-s)^{\alpha-1} g(s,y(s))ds \approx \int_{0}^{t_{n+1}} (t_{n+1}-s)^{\alpha-1} \bar{g}_{n+1}(s)ds,
\end{equation}
where now the $\bar{g}_{n+1}$ is the piecewise linear interpolant for $g$ with nodes and knots chosen at the nodes $t_{j}$. Hence,
\begin{equation}
\int_{0}^{t_{n+1}} (t_{n+1}-z)^{\alpha-1} \bar{g}_{n+1}(z)dz = \frac{h^{\alpha}}{\alpha(\alpha+1)} \sum_{j=0}^{n+1} a_{j,n+1} g(t_{j}),
\end{equation}
where
\begin{equation}
a_{j,n+1}=\begin{cases}
n^{\alpha+1}-(n-\alpha)(n+1)^{\alpha}, & j=0, \\
(n-j+2)^{\alpha+1}+(n-j)^{\alpha+1}-2(n-j+1)^{\alpha+1}, & 1 \leq j \leq n, \\
1, & j=n+1,
\end{cases}
\end{equation}
which can be compared with \cite{Diethelm}. As a consequence, we obtain the corrector formula
\begin{multline}\label{eq:corrector}
y_{h}(t_{n+1}) =y(0)+ \frac{1-\alpha}{B(\alpha)} g(t_{n+1},y_{h}(t_{n+1})) \\ + \frac{\alpha\,h^{\alpha}}{B(\alpha) \Gamma(\alpha+2)} \left( g(t_{n+1},y_{h}^{P}(t_{n+1}))+\sum_{j=0}^{n} a_{j,n+1} g(t_{j},y_{h}(t_{j})) \right).
\end{multline}
In order to obtain the predictor formula, we follow again \cite{Diethelm} and use the product rectangle rule
\begin{equation}
\int_{0}^{n+1} (t_{n+1}-z)^{\alpha-1} g(z)dz \approx  \frac{h^{\alpha}}{\alpha}  \sum_{j=0}^{n}((n+1-j)^{\alpha}-(n-j)^{\alpha}) g(t_{j}),
\end{equation}
in order to obtain
\begin{equation}\label{eq:predictor}
y_{h}^{P}(t_{n+1})=y(0)+ \frac{1-\alpha}{B(\alpha)} g(t_{n+1},y_{h}(t_{n+1})) 
+ \frac{h^{\alpha}}{B(\alpha) \Gamma(\alpha)}   \sum_{j=0}^{n}((n+1-j)^{\alpha}-(n-j)^{\alpha}) g(t_{j},y_{h}(t_{j})). 
\end{equation}

\subsection{Example 1}

Let us consider the initial value problem
\begin{equation}\label{eq:ex1}
\begin{cases}
\displaystyle{\fdabc{0}{\alpha} y(t)=t, \quad t \in [0,1]}, \\[2mm]
y(0)=y_{0}.
\end{cases}
\end{equation}
The explicit solution of the above fractional initial value problem is 
\begin{equation}\label{eq:solex1}
y(t)=\frac{\alpha  \left(\alpha  (\alpha +1) y(0)+t^{\alpha +1}\right)-(\alpha -1) \Gamma (\alpha +2) (y(0)+t)}{\alpha  (\alpha +1) (\alpha -\alpha \Gamma (\alpha )+\Gamma (\alpha ))},
\end{equation}
which can be obtained by using \eqref{laplace:caputo}. Notice that as $\alpha \to 1$ \eqref{eq:solex1} goes formally to $t^{2}/2+y(0)$. In the particular case $\alpha=1/2$ then \eqref{eq:solex1} reduces to
\begin{equation}
y(t)=y_{0}+\frac{\frac{4 t^{3/2}}{\sqrt{\pi }}+3 t}{6 B(1/2)}.
\end{equation}
 If we consider $10$ equidistant nodes in $[0,1]$, from \eqref{eq:predictor} we obtain for $\alpha=1/2$ and the initial condition $y(0)=0$ the following table of predicted values in the interval $[0,1]$,
\begin{multline*}
\{0., 0.0639309, 0.150674, 0.246866, 0.350309, 0.459864, 0.574804, \\
0.694613, 0.818899, 0.947352, 1.07972\}.
\end{multline*}
Moreover, from \eqref{eq:corrector} for the corrected values we have
\begin{multline*}
\{0., 0.079139, 0.170877, 0.270816, 0.377388, 0.489686, 0.607097, \\
0.729174, 0.855567, 0.985996, 1.12023\}.
\end{multline*}
In Figure \ref{fig:fig1} we show the exact solution and its approximation by using the predictor-corrector method proposed. We would like to notice that for larger number of nodes in $[0,1]$ the approximated solution coincides with the exact solution in the whole interval.
\begin{figure}[ht!]
\centering 
\includegraphics[width=0.4\textwidth]{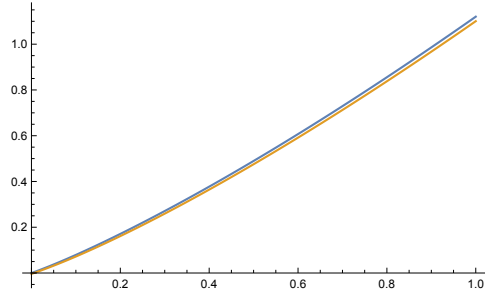}
\caption{Exact solution of \eqref{eq:ex1} ---blue--- and its approximation ---orange--- by using the predictor-corrector method proposed.}
\label{fig:fig1}
\end{figure}

\subsection{Example 2}
Let us now consider the fractional initial value problem
\begin{equation}\label{eq:ex3}
\begin{cases}
\displaystyle{\fdabc{0}{\alpha} y(t)=\exp(-t\,y), \quad t \in [0,1]}, \\[2mm]
y(0)=y_{0}.
\end{cases}
\end{equation}
If we consider $20$ equidistant nodes in $[0,1]$, from \eqref{eq:predictor} we obtain for $\alpha=9/10$ and the initial condition $y(0)=1$ the following table of predicted values in the interval $[0,1]$,
\begin{multline*}
\{ 1.16712, 1.21523, 1.25704, 1.29398, 1.32682, 1.35609, 1.38222, \\
1.40554, 1.42635, 1.44492, 1.46148, 1.47623, 1.48936, 1.50104, \\
1.51141, 1.52061, 1.52876, 1.53596, 1.54231, 1.54789 \}.
\end{multline*}
Moreover, from \eqref{eq:corrector} for the corrected values we have
\begin{multline*}
\{ 1.16513, 1.21157, 1.25184, 1.28737, 1.31892, 1.34702, 1.37209, \\
1.39447, 1.41445, 1.43229, 1.4482, 1.46239, 1.47504, 1.4863, 1.49631, \\
1.5052, 1.51309, 1.52007, 1.52624, 1.53168 \}.
\end{multline*}
For computing these values (predictor and corrector) we have numerically solved for each step $k$ the implicit equations given by the scheme. Notice that if we consider the differential equation $y'=\exp(-t y)$ with initial condition $y(0)=1$ then numerically we have $y(1)=1.54153$.

In Figure \ref{fig:fig3} we show the approximation of the solution of \eqref{eq:ex3} found by using the predictor-corrector method proposed.
\begin{figure}[ht!]
\centering 
\includegraphics[width=0.4\textwidth]{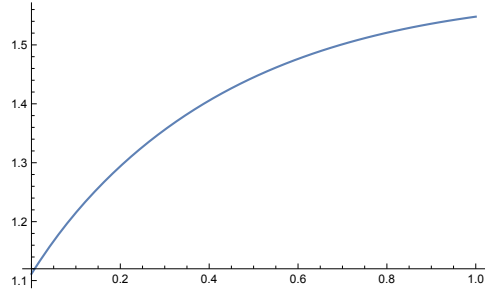}
\caption{Approximation of the solution of \eqref{eq:ex3} by using the predictor-corrector method proposed.}
\label{fig:fig3}
\end{figure}

\subsection{Example 3}
Let us consider the fractional initial value problem
\begin{equation}\label{eq:ex4}
\begin{cases}
\displaystyle{\fdabc{0}{\alpha} y(t)=y, \quad t \in [0,1]}, \\[2mm]
y(0)=y_{0}.
\end{cases}
\end{equation}
First of all, we shall solve explicitly this fractional initial value problem by using Laplace transform. In doing so, if we apply the Laplace transform to the first equation of \eqref{eq:ex4} we get
\[
\frac{B(\alpha)}{1-\alpha} \frac{p^{\alpha} {\mathcal{L}}(f(t))(p)-p^{\alpha-1} y_{0}}{p^{\alpha}+\alpha/(1-\alpha)} = {\mathcal{L}}(f(t))(p),
\]
which implies
\[
{\mathcal{L}}(f(t))(p) = \frac{B(\alpha) p^{\alpha-1}}{p^{\alpha}(B(\alpha)+\alpha-1)-\alpha} y_{0}. 
\]
As a consequence, we obtain that
\begin{equation}\label{eq:gsm}
f(t)=\frac{B(\alpha)}{B(\alpha)+\alpha-1} y_{0} E_{\alpha} \left[ \frac{\alpha}{B(\alpha)+\alpha-1} t^{\alpha} \right]
\end{equation}
is the explicit solution of \eqref{eq:ex4}, where $E_{\alpha}(t)$ denotes the Mittag-Leffler function defined in \eqref{eq:mlf}.

By using the numerical scheme proposed in this article, with $\alpha=9/10$, $y_{0}=1$, and 20 points of discretization, we obtain the following table of numerical values for the solution of \eqref{eq:ex4} with the predictor
\begin{multline*}
\{1.1937, 1.27314, 1.35327, 1.43591, 1.52179, 1.6114, 1.70514, \\
1.80337, 1.90641, 2.01459, 2.12824, 2.24771, 2.37333, 2.50546, \\
2.64448, 2.79077, 2.94474, 3.10682, 3.27746, 3.45713 \},
\end{multline*}
and also the following values for the solution of \eqref{eq:ex4} with the corrector
\begin{multline*}
\{1.20134, 1.28315, 1.36621, 1.45204, 1.54137, 1.63472, 1.73253, \\
1.83515, 1.94297, 2.05632, 2.17559, 2.30113, 2.43333, 2.57257, \\
2.71929, 2.87389, 3.03685, 3.20863, 3.38973, 3.58067\}.
\end{multline*}
In Figure \ref{fig:fig4} we show the explicit solution of \eqref{eq:ex4} given by \eqref{eq:gsm} and the approximation of the solution of \eqref{eq:ex4} found by using the predictor-corrector method proposed. Moreover, in Figure \ref{fig:fig4bis} we show the explicit solution of \eqref{eq:ex4} given by \eqref{eq:gsm} and the approximation of the solution of \eqref{eq:ex4} found by using the predictor-corrector method proposed by considering $h=0.01$ as step size for the mesh, with $\alpha=9/10$, and $y_{0}=1$.
\begin{figure}[ht!]
\centering 
\includegraphics[width=0.4\textwidth]{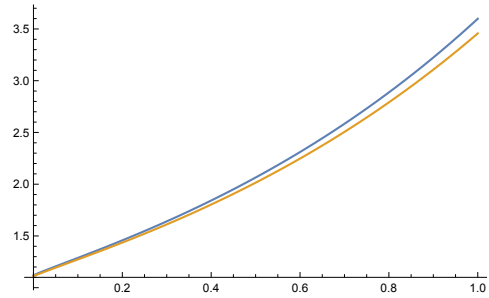}
\caption{Explicit solution of \eqref{eq:ex4} given by \eqref{eq:gsm} and the approximation of the solution of \eqref{eq:ex4} found by using the predictor-corrector method proposed, by considering $h=0.05$ as step size for the mesh, with $\alpha=9/10$, and $y_{0}=1$}
\label{fig:fig4}
\end{figure}

\begin{figure}[ht!]
\centering 
\includegraphics[width=0.4\textwidth]{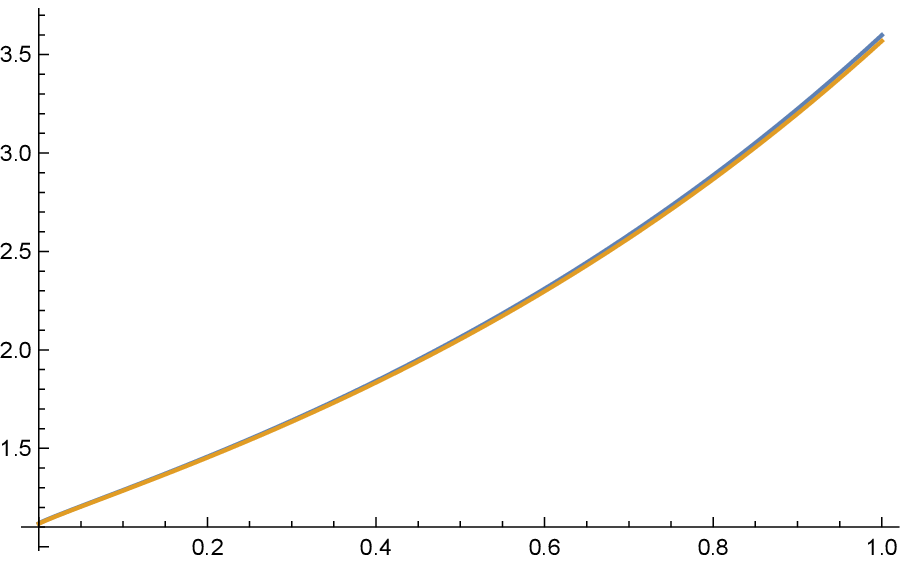}
\caption{Explicit solution of \eqref{eq:ex4} given by \eqref{eq:gsm} and the approximation of the solution of \eqref{eq:ex4} found by using the predictor-corrector method proposed, by considering $h=0.01$ as step size for the mesh, with $\alpha=9/10$, and $y_{0}=1$.}
\label{fig:fig4bis}
\end{figure}

\subsection{Example 4}
Let us consider the fractional logistic equation
\begin{equation}\label{eq:ex5}
\begin{cases}
\displaystyle{\fdabc{0}{\alpha} y(t)=r y(1-y), \quad t \in [0,1]}, \\[2mm]
y(0)=y_{0}.
\end{cases}
\end{equation}
In this case, we shall fix a value of $\alpha$ close to one in order to compare the numerical solution with the exact solution, which is just known in the classical situation \cite{fraceq}. Let us fix $r=-5$, $\alpha=99/100$, $y_{0}=1/2$, and step size $h=0.01$. Then, the numerical solution and the exact solution (of the classical case) are compared in Figure \ref{fig:fig5}.
\begin{figure}[ht!]
\centering 
\includegraphics[width=0.4\textwidth]{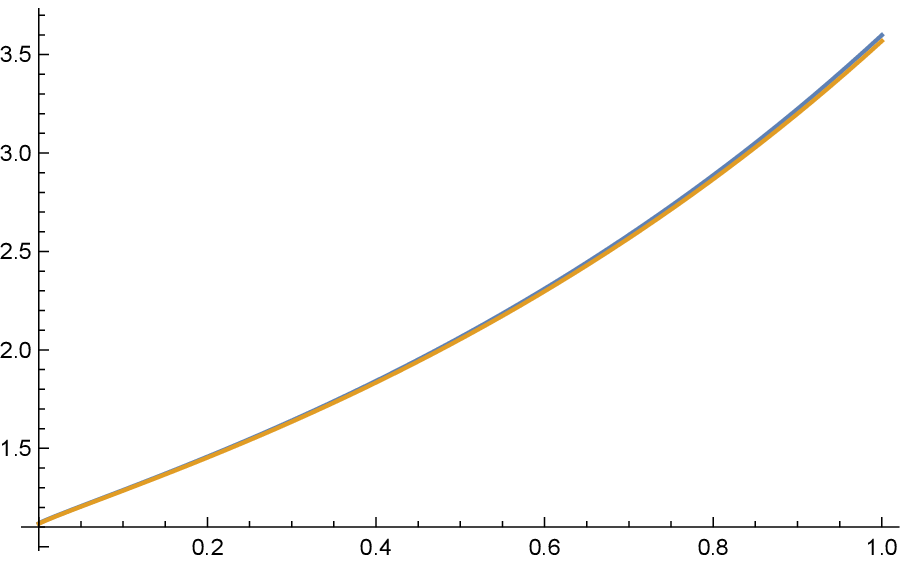}
\caption{Explicit solution of the classical logistic equation and the approximation of the solution of \eqref{eq:ex5} found by using the predictor-corrector method proposed, by considering $h=0.01$ as step size for the mesh, with $\alpha=99/100$, $r=-5$, and $y_{0}=1/2$.}
\label{fig:fig5}
\end{figure}

\subsection{Example 5}

Let us consider the following fractional analogue of the Lotka-Volterra equations
\begin{equation}\label{eq:ex6}
\begin{cases}
\displaystyle{\fdabc{0}{\alpha} x(t)=ax-b x y, \quad t \in [0,1]}, \\[2mm]
\displaystyle{\fdabc{0}{\alpha} y(t)=-cy+ d x y, \quad t \in [0,1]}, \\[2mm]
x(0)=x_{0}, \quad y(0)=y_{0}.
\end{cases}
\end{equation}
In this case, we shall compare the numerical solution with the solution of the classical (non fractional) case for some values of $\alpha$. Let us fix $a=1$, $b=2$, $c=3$, $d=4$, $x_{0}=y_{0}=1$, and step size $h=0.01$. Then, the solutions of the non fractional and fractional cases are compared in Figures \ref{fig:fig6}-\ref{fig:fig8}.
\begin{figure}[ht!]
\centering 
\begin{center}
  \includegraphics[width=0.25\textwidth]{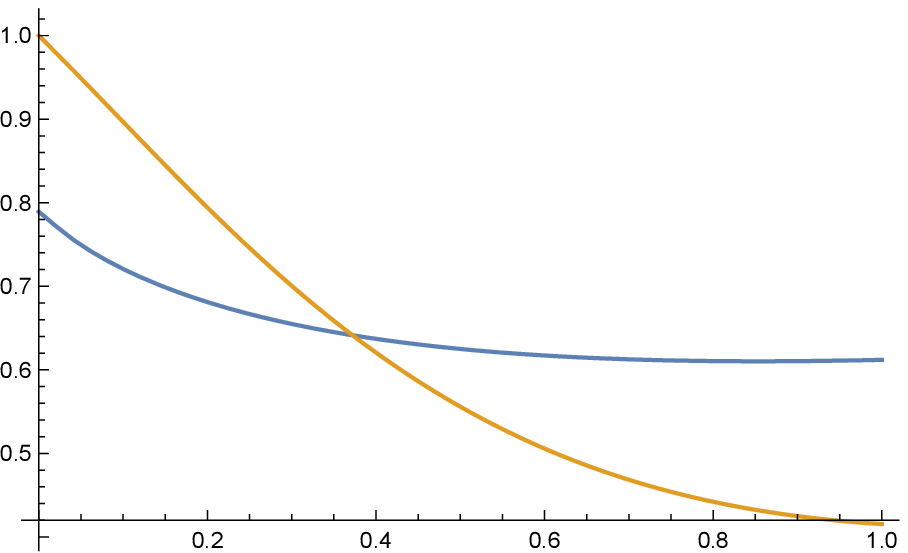} $\quad$
  \includegraphics[width=0.25\textwidth]{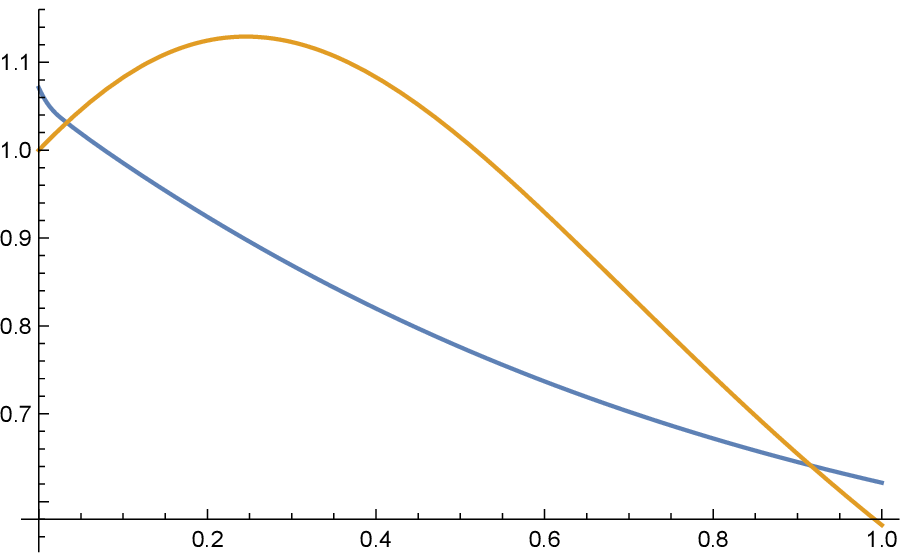}
\end{center}
\caption{Comparison between the numerical solution of \eqref{eq:ex6} and the solution of the classical (non fractional) case for the specific values of the parameters $a=1$, $b=2$, $c=3$, $d=4$, $x_{0}=y_{0}=1$, and step size $h=0.01$ with $\alpha=4/5$. }
\label{fig:fig6}
\end{figure}

\begin{figure}[ht!]
\centering 
\begin{center}
  \includegraphics[width=0.25\textwidth]{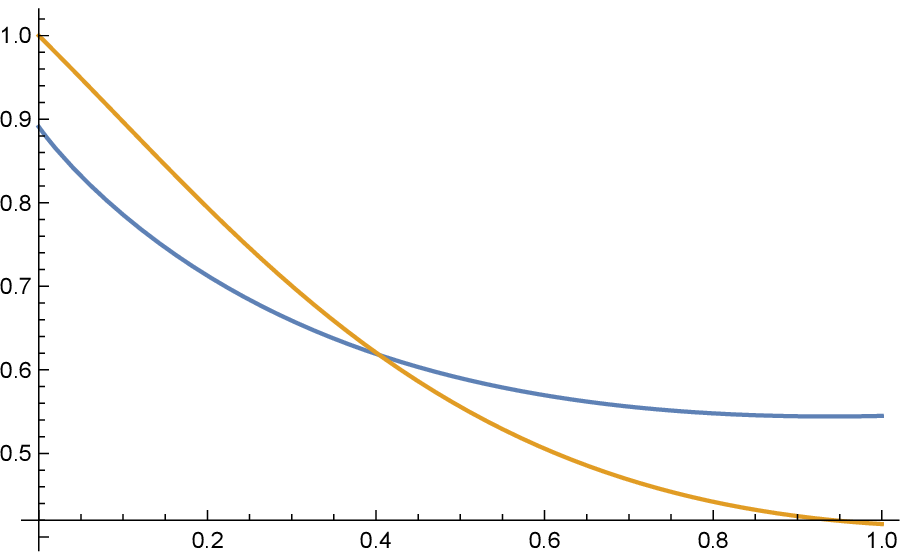} $\quad$
  \includegraphics[width=0.25\textwidth]{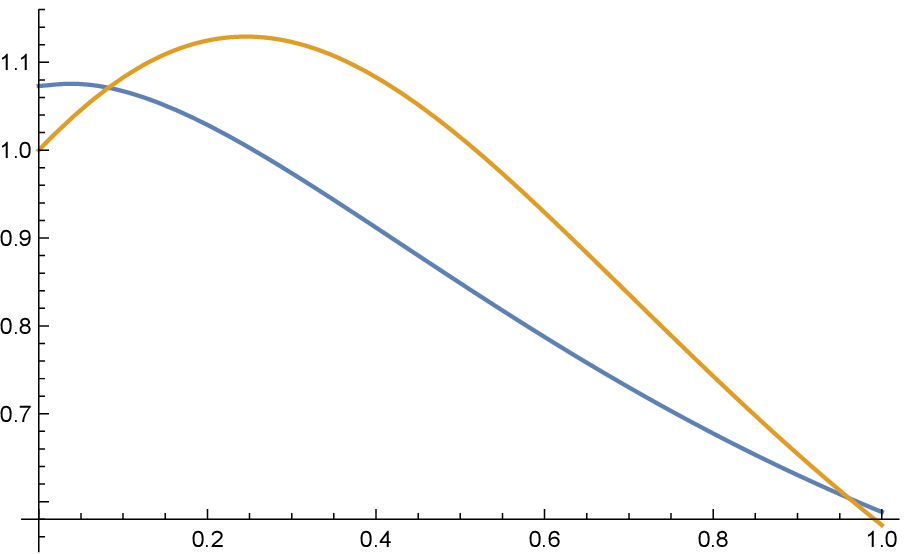}
\end{center}
\caption{Comparison between the numerical solution of \eqref{eq:ex6} and the solution of the classical (non fractional) case for the specific values of the parameters $a=1$, $b=2$, $c=3$, $d=4$, $x_{0}=y_{0}=1$, and step size $h=0.01$ with $\alpha=9/10$. }
\label{fig:fig7}
\end{figure}

\begin{figure}[ht!]
\centering 
\begin{center}
  \includegraphics[width=0.25\textwidth]{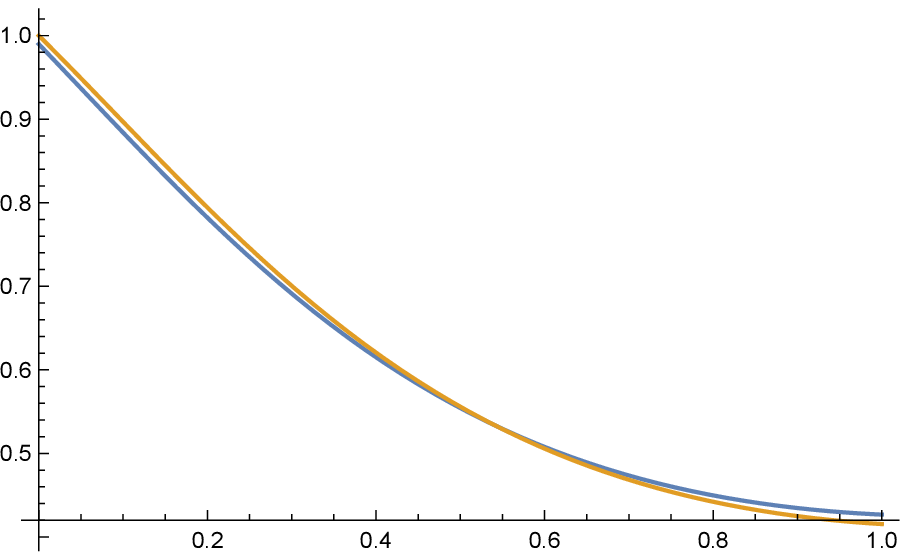} $\quad$
  \includegraphics[width=0.25\textwidth]{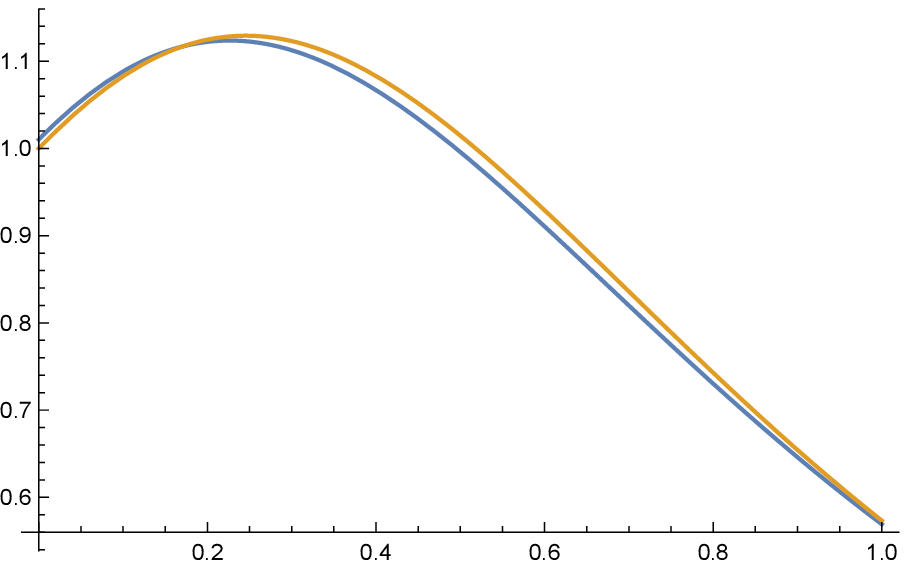}
\end{center}
\caption{Comparison between the numerical solution of \eqref{eq:ex6} and the solution of the classical (non fractional) case for the specific values of the parameters $a=1$, $b=2$, $c=3$, $d=4$, $x_{0}=y_{0}=1$, and step size $h=0.01$ with $\alpha=99/100$. }
\label{fig:fig8}
\end{figure}

\newpage

\section*{Conclusion}
In this paper we have presented a numerical scheme to solve fractional initial value problems involving the Atangana-Baleanu fractional derivative. Some examples have been presented in order to show how the method works in different situations, running from very simple fractional initial value problems, fractional extensions of $y'=t$ or $y'=y$, to fractional logistic equation, for which the exact solution is not known. A fractional analogue of the Lotka-Volterra equations has been also considered. It should be finally mentioned that our goal here is not to exploit all possible situations covered by this numerical scheme, but to emphasize that it can be used to solve numerically many fractional initial value problems. Further research related with the accuracy of the method is now under analysis.

\end{document}